\documentclass[12pt, reqno]{amsart}

\title{Representations of cones and applications to decision theory}

\usepackage[T1]{fontenc}
\usepackage{amsmath}
\usepackage{amssymb}
\usepackage{amsthm}
\usepackage[left=3.5cm, right=3.5cm, top=3.2cm, bottom=3.2cm]{geometry}
\usepackage{hyperref}
\usepackage{fancyhdr}
\usepackage{enumitem}
\usepackage{comment}
\usepackage{nicefrac}
\usepackage{bm}
\usepackage{mathrsfs}
\usepackage{graphicx}
\usepackage[utf8]{inputenc}
\usepackage{cancel}
\usepackage{mathtools}
\usepackage{tikz}
\usetikzlibrary{cd,decorations.pathreplacing,matrix,arrows,positioning,automata,shapes,shadows,calc,fadings,decorations,snakes,through,intersections}
\usepackage{float}

\AtBeginDocument{%
   \def\MR#1{}
}

\newtheorem{thm}{Theorem}[section]
\newtheorem{cor}[thm]{Corollary}
\newtheorem{lem}[thm]{Lemma}
\newtheorem{prop}[thm]{Proposition}

\theoremstyle{definition} 
\newtheorem{defi}[thm]{Definition}
\let\olddefi\defi
\renewcommand{\defi}{\olddefi\normalfont}
\newtheorem{question}{Question}
\let\oldquestion\question
\renewcommand{\question}{\oldquestion\normalfont}
\newtheorem{example}[thm]{Example}
\let\oldexample\example
\renewcommand{\example}{\oldexample\normalfont}
\newtheorem{rmk}[thm]{Remark}
\let\oldrmk\rmk
\renewcommand{\rmk}{\oldrmk\normalfont}

\usepackage{todonotes}

\hypersetup{
    pdftitle={Prova},
    pdfauthor={},
    pdfmenubar=false,
    pdffitwindow=true,
    pdfstartview=FitH,
    colorlinks=true,
    linkcolor=blue,
    citecolor=green,
    urlcolor=cyan
}

\providecommand{\MR}[1]{}

\providecommand{\MR}{\relax\ifhmode\unskip\space\fi MR }

\providecommand{\href}[2]{#2}

\begin{document}

\author[P.~Leonetti]{Paolo Leonetti}
\address[P.~Leonetti]{Department of Decision Sciences, Universit\`a Luigi Bocconi, via Roentgen 1, Milan 20136, Italy}
\email{leonetti.paolo@gmail.com}

\author[G.~Principi]{Giulio Principi}
\address[G.~Principi]{Department of Economics, New York University, 19 West 4th Street, New York 10003, USA}
\email{gp2187@nyu.edu}

\thanks{P.~Leonetti is grateful to PRIN 2017 (grant 2017CY2NCA) for financial support. G.~Principi is grateful to MacCracken Fellowship.}


\begin{abstract} 
\noindent 
Let $C$ be a cone in a locally convex Hausdorff topological vector space $X$ containing $0$. 
We show that there exists a (essentially unique) nonempty family $\mathscr{K}$ of nonempty subsets of the topological dual $X^\prime$ such that 
$$
C=\{x \in X: \forall K \in \mathscr{K}, \exists f \in K, \,\, f(x) \ge 0\}.
$$
Then, we identify the additional properties on the family $\mathscr{K}$ which characterize, among others, closed convex cones, open convex cones, closed cones, and convex cones. 
For instance, if $X$ is a Banach space, then $C$ is a closed cone if and only if the family $\mathscr{K}$ can be chosen with nonempty convex compact sets. 

These representations provide abstract versions of several recent results in decision theory and give us the proper framework to obtain new ones. 
This allows us to characterize preorders which satisfy the independence axiom over certain probability measures, answering an open question in [Econometrica~\textbf{87} (2019), no. 3, 933--980].
\end{abstract}
\maketitle
\thispagestyle{empty}

%
\section{Introduction}\label{sec:intro}

The study of the geometry of cones and their characterization is central in many fields of pure and applied mathematics. 
For instance, they are related to the algebraic and topological properties of ordered normed vector spaces, Riesz spaces, and their linear operators, see e.g. \cite{
MR2262133}. 
They appear surprisingly also in the theory of Banach spaces: e.g., Milman and Milman proved in \cite{MR0171152} that a Banach space 
is not reflexive if and only if it contains a copy of the positive cone of $\ell_1$, cf. also \cite{MR2754997, MR2577803}; and, with a different behaviour, a Banach space $X$ contains a copy of $c_0$ if and only if it contains a copy the positive cone of $c_0$, see \cite{MR2861605}. 
In mathematical economics, 
the geometry of ordering cones is crucial in general equilibrium theory \cite{MR1075992} and in decision theory, where they have been implicitly used to characterize certain type of preferences between lotteries \cite{Dubra_et_al, MR3182925, MR3957335}, building on the idea of Aumann's cone \cite{MR174381}. 
In addition, cones appeared in vector-valued convex optimization \cite{MR1994718, MR1116766}  
and in finance to characterize the no-arbitrage property \cite{MR3822757}. 
For several other cones-related results, see e.g. \cite{MR4223027, MR3365863, MR4151084, MR1666566, MR4273125}. 

The main result of this work is to provide a characterization of cones $C$ in arbitrary 
locally convex 
Hausdorff 
topological vector spaces $X$, and to show that such representation is unique in a precise sense, see Theorem \ref{thm:dualconesgeneralized} and \ref{thm:uniqueness} in Section~\ref{sec:mainresults}. The representation will be of the type: \textquotedblleft There exists a nonempty family $\mathscr{K}$ of nonempty subsets of the dual $X^\prime$ such that $x\in C$ if and only if
$$
\forall K \in \mathscr{K}, \exists f\in K,\,\, f(x) \ge 0. \text{\textquotedblright}
$$
A 
result on the same spirit 
appeared in \cite[Theorem 1.b]{MR3566440}, where Nishimura and Ok proved that, if $\mathcal{R}$ is a binary relation on a metric space $S$ (that is, $\mathcal{R}\subseteq S^2$), then $\mathcal{R}$ is reflexive if and only if there exists a nonempty family $\mathscr{U}$ of nonempty subsets of $\mathrm{C}(S):=\{u: S\to \mathbf{R} \text{ continuous}\}$ such that, for each $a,b \in S$, 
we have 
$(a,b) \in \mathcal{R}$ if and only if
\begin{equation}\label{eq:nishimuraok}
\forall U \in \mathscr{U}, \exists u \in U, \quad u(a) \ge u(b).
\end{equation}
Characterizations analogous to \eqref{eq:nishimuraok} appeared in \cite[Theorem 1]{MR3957335} in the context of preferences over lotteries and in \cite[Theorem 2]{MR2888839} for preferences over Anscombe--Aumann acts. 
On a different direction, Hausner and Wendel proved in \cite{MR52045} that every totally ordered vector space is isomorphic to a certain subset of some lexicographic-ordered function space.

Then, we proceed to Section \ref{sec:additionalproperties} where we characterize cones with additional properties. More precisely, our main results provide representations for:
\begin{enumerate}[label={\rm (\roman{*})}]
\item cones $C\subseteq X$ such that $C\cup (-C)=X$ (Theorem \ref{thm:propertycompleteness}); 
\item convex cones (Theorem \ref{thm:convexcones});
\item closed convex cones and their interiors (Theorem \ref{thm:dualdualcone});
\item open convex cones, provided that $X$ is normed (Proposition \ref{prop:openconvex}); 
\item closed cones, provided that $X$ is Banach (Theorem \ref{thm:closedcone}).
\end{enumerate}

These characterizations provide abstract versions of several results in decision theory contained, e.g., in \cite{Dubra_et_al, MR2398816, MR3182925, MR3957335}. 
In addition, we conclude in Section \ref{sec:appl} with some applications. 
In Corollary \ref{cor:firstresultok} we recover the first main result in \cite{MR3957335}. 
Furthermore, in such context, we characterize the additional structure inherited from the transitivity axiom, which answers an open question in \cite{MR3957335}, see Corollary \ref{cor:transitivity}. 
Lastly, we characterize reflexive binary relations which satisfy the independence axiom in the context of Anscombe--Aumann acts, providing a partial answer to another open question in \cite{MR3957335}, see Corollary \ref{cor:AnscombeAumann}.




\section{Main Results}\label{sec:mainresults}
Given a topological vector space (tvs) $X$ over the real field, we denote by $X^\prime$ its topological dual. 
A nonempty subset $C\subseteq X$ is said to be a \emph{cone} if $\lambda C\subseteq C$ for all positive reals $\lambda > 0$. A cone $C\subseteq X$ is said to be \emph{pointed} if $0\in C$, otherwise it is called \emph{blunted}. 
Hence, both $X$ and $\{0\}$ are pointed convex cones, and $X$ is the unique open pointed cone. 
Let $\mathscr{C}(X)$ be the family of closed convex cones of $X$. 
Lastly, given a nonempty set $S\subseteq X$, let $\mathrm{cone}(S):=\{\lambda x: \lambda>0, x \in S\}$ be the smallest cone containing $S$, and $\mathrm{co}(S)$ and $\overline{\mathrm{co}}(S)$ be the convex hull and closed convex hull of $S$, respectively. 

A pair of vector spaces $(X,Y)$ is a \emph{dual pair} if there exists a bilinear map $\langle \cdot,\cdot \rangle: X\times Y\to \mathbf{R}$ such that the families $\{\langle \cdot, y\rangle: y \in Y\}$ and $\{\langle x, \cdot\rangle: x \in X\}$ separate the points of $X$ and $Y$, respectively. Such map $\langle \cdot,\cdot \rangle$ is called \emph{duality}. 
Unless otherwise stated, $X$ and $Y$ are endowed with the weak topology $\sigma(X,Y)$ and the weak$^\star$ topology $\sigma(Y,X)$, respectively. 
Accordingly, it is well known that 
both $X$ and $Y$ are locally convex Hausdorff tvs, and that $X^\prime=Y$ and $Y^\prime=X$, 
up to isomorphisms; 
conversely, if $X$ is a 
locally convex Hausdorff 
tvs, then $(X,X^\prime)$ is a dual pair, 
see e.g. \cite[Section 5.14]{MR2378491} or \cite[Section 8.2]{MR2317344}. 

\emph{Hereafter, 
unless otherwise stated, 
we refer always implicitly to the dual pair $(X,Y)$ with $Y=X^\prime$. E.g., we write \textquotedblleft a subset $S\subseteq X$ is closed\textquotedblright\ for \textquotedblleft a subset $S\subseteq X$ is closed in the weak topology  $\sigma(X,Y)$.\textquotedblright}

We present now the main definition of this work. 
\begin{defi}\label{defi:base}
Given a nonempty family $\mathscr{K}\subseteq \mathcal{P}(Y)$ of nonempty sets, 
we say that a cone $C\subseteq X$ \emph{is represented by $\mathscr{K}$} if $C=C_{\mathscr{K}}$, where
\begin{equation}\label{eq:claimedequalityconedefi}
C_{\mathscr{K}}:=\left\{x \in X: \forall K \in \mathscr{K}, \exists y \in K, \,\, \langle x,y\rangle \ge 0\right\}. 
\end{equation}
By convention, we assume that $X$ is represented by 
$\{\emptyset\}$. 
\end{defi}
Clearly, if the cone $C$ admits such a representation, then $C$ is pointed. 
Conversely, as we are going to show in a moment every pointed cone is represented by some family $\mathscr{K}$. Note that, if $C$ is blunted, then $X\setminus C$ is a pointed cone; hence, passing to the complement or simply subtracting $0$, the pointedness assumption is not restrictive of generality. 

At this point, given a pointed cone $C\subseteq X$, define 
$$
\mathscr{K}(C):=\{G_x: x\notin C\},
$$
where
$$
\forall x \in X, \quad G_x:=\{y \in Y: \langle x,y\rangle <0\}.
$$
Observe that, for each nonzero $x \in X$, $G_x$ is a nonempty open convex blunted cone (hence, relative to the weak$^\star$ topology on $Y$). 
In addition, $G_{x_1}=G_{x_2}$, for some nonzero $x_1,x_2 \in X$ if and only if $\mathrm{cone}(\{x_1\})=\mathrm{cone}(\{x_2\})$. Indeed, in the opposite, $\{x_1,x_2\}$ would be a linearly independent family and, if $\mathscr{B}$ is an Hamel basis of $X$ containing $\{x_1,x_2\}$, then the unique functional $y \in Y$ supported on $\mathrm{span}\{x_1,x_2\}$ and such that  $\langle x_1,y\rangle=-\langle x_2,y\rangle=-1$ belongs to $G_{x_1}\setminus G_{x_2}$. Lastly, given pointed cones $C_1,C_2\subseteq X$, we have $C_1\subseteq C_2$ if and only if $\mathscr{K}(C_2)\subseteq \mathscr{K}(C_1)$. 

Our first main result follows:
\begin{thm}\label{thm:dualconesgeneralized}
Let $C\subseteq X$ be a pointed cone. 
Then $C$ is represented by $\mathscr{K}(C)$. 
\end{thm}
 
\begin{proof}
The claim holds if $C=X$, hence suppose hereafter that $C$ is proper. 
Set $D:=C_{\mathscr{K}(C)}$ and suppose that there exists a vector $x \in D\setminus C$. Since $x\notin C$, then $G_x$ belongs to $\mathscr{K}(C)$. Given that $x \in D$, there exists $y \in G_x$ such that $\langle x,y\rangle \ge 0$. However, this contradicts the definition of $G_x$. Therefore 
$D\subseteq C$. 

Conversely, pick two vectors $x \in C$ and $x_0\notin C$, and define the closed convex cone $C_0:=\{\lambda x: \lambda \ge 0\}$. Since $C$ is a pointed cone, then $C_0\subseteq C$, hence $x_0\notin C_0$. 
Thanks to the 
Strong Separating Hyperplane Theorem, 
see e.g. \cite[Theorem 8.17]{MR2317344}, 
there exists a nonzero $y \in Y$ strongly separating $C_0$ and $x_0$, meaning explicitly that
$$
\exists \alpha \in \mathbf{R}, \forall \lambda \ge 0, \quad 
\langle \lambda x, y\rangle \ge \alpha > \langle x_0,y\rangle.
$$
However, this is possible only if $\alpha \le 0$, so that $\langle x_0,y\rangle <0$. 
Observe that the value $\langle x,y\rangle$ cannot be negative: indeed, in such case, this would contradict the inequality $\langle \lambda x,y\rangle>\langle x_0,y\rangle$ with $\lambda=2\langle x_0,y\rangle / \langle x,y\rangle$.  
Hence $y \in G_{x_0} \in \mathscr{K}(C)$ and $\langle x,y\rangle \ge 0$, which proves that $C\subseteq D$. 
\end{proof}

The representation of a cone $C$ is certainly not unique. 
As an extreme 
case, the whole space $X$ is represented by the singleton $\{\{y,-y\}\}$  for every $y \in Y$. This motivates the following definition: 
\begin{defi}\label{defi:trivial}
A nonempty set $K\subseteq Y$ is said to be \emph{trivial} if 
$$
\forall x \in X, \exists y \in K, \,\,
\langle x,y\rangle \ge 0.
$$
The family of trivial sets is denoted by $\mathscr{K}_{\sharp}$. 
\end{defi}
Equivalently, $X$ is represented by a family $\mathscr{K}$ if and only if $\mathscr{K}\subseteq \mathscr{K}_\sharp$. 
Note that $\mathscr{K}_\sharp$ is an upward directed family which is not a filter: indeed both $\{\{0\}\}$ and $\{Y\setminus \{0\}\}$ belong to $\mathscr{K}_\sharp$, while $\mathscr{K}_\sharp \neq \mathcal{P}(Y)$. 
In addition, the definition of trivial set is strictly related to the notion of one-sided K\"{o}the polar, so that $K \in \mathscr{K}_\sharp$ if and only if $\{x \in X: \forall y \in K, \,\, \langle x,y\rangle <0\}=\emptyset$, cf. \cite[p. 215]{MR2378491}. 
Note also that, for each nonzero $x\in X$, $G_x$ is not trivial in the sense of Definition \ref{defi:trivial} since $x \notin C_{\{G_x\}}$, which implies that 
$$
\mathscr{K}(C) \cap \mathscr{K}_\sharp=\emptyset
$$
for all cones $C\subseteq X$. 

Then, we have the following easy [non-]uniqueness representation, cf. Theorem \ref{thm:uniqueness} below, which we state mainly for practical purposes: 
\begin{cor}\label{cor:easyrepresentation}
Let $C\subseteq X$ be a proper pointed cone, and $\mathscr{K}\subseteq \mathcal{P}(Y)$ be a nonempty family such that 
$
\mathscr{K}(C)=\{\mathrm{co}\left(\mathrm{cone}(K)\right): K \in \mathscr{K}\setminus \mathscr{K}_\sharp\}.
$ 
Then $C$ is represented by $\mathscr{K}$. 
\end{cor}
\begin{proof}
Note that $C_{\mathscr{K}}=C_{\mathscr{K}\setminus \mathscr{K}_\sharp}$. In particular, we can assume without loss of generality that $0\notin K$ for all $K \in \mathscr{K}$. 
Fix a nonzero $x\in C_{\mathscr{K}(C)}$ and a set $K \in \mathscr{K}$. Given the standing hypothesis, there exists $x_0\notin C$ such that $\mathrm{co}(\mathrm{cone}(K))=G_{x_0}$. 
By the definition of $C_{\mathscr{K}(C)}$, 
there exists $y_0 \in G_{x_0}$ such that $\langle x,y_0\rangle \ge 0$. 
However, since $\mathrm{co}(\mathrm{cone}(K))=G_{x_0}$, there exist positive reals $\lambda_1,\ldots,\lambda_k>0$ and vectors $y_1,\ldots,y_k \in K$ such that $y_0=\sum_{i=1}^k \lambda_i y_i$. 
By the linearity of $\langle \cdot, \cdot\rangle$, we conclude that $\langle x,y_i\rangle \ge 0$ for some $i \in \{1,\ldots,k\}$. Therefore $C_{\mathscr{K}(C)}\subseteq C_{\mathscr{K}}$. 

Conversely, suppose that $x\notin C_{\mathscr{K}(C)}$. By Theorem \ref{thm:dualconesgeneralized}, there exists $x_0\notin C$ such that $\langle x,y_0\rangle <0$ for all $y_0 \in G_{x_0}$. Pick $K \in \mathscr{K}$ with $\mathrm{co}(\mathrm{cone}(K))=G_{x_0}$. In particular, $\langle x,y\rangle <0$ for all $y \in K$. Therefore $C_{\mathscr{K}}\subseteq C_{\mathscr{K}(C)}$ and, hence, they coincide. 
The conclusion follows by Theorem \ref{thm:dualconesgeneralized}, so that $C=C_{\mathscr{K}(C)}=C_{\mathscr{K}}$. 
\end{proof}

The above representations 
are, of course, not the unique ones. As an example, consider the dual pair $(\mathbf{R}^2, \mathbf{R}^2)$. Then the positive cone $C=[0,\infty)^2$ is represented by the family $\mathscr{K}=\left\lbrace\{(\cos \alpha, \sin \alpha)\}: \alpha \in \left[0,\frac{\pi}{2}\right]\right\rbrace$, cf. Theorem \ref{thm:dualdualcone} below. However, for each $K \in \mathscr{K}$, the cone $\mathrm{co}(\mathrm{cone}(K))$ has empty interior, hence it cannot be a member of $\mathscr{K}(C)$, since the latter contains only nonempty open sets. 
%


In a sense, we are going to improve Corollary \ref{cor:easyrepresentation} by showing that the \textquotedblleft largest\textquotedblright\ representation $\mathscr{K}$ of a cone $C\subseteq X$ is precisely $\mathscr{K}(C)$. 
For, given a nonempty family $\mathscr{K}\subseteq \mathcal{P}(Y)$ of nonempty sets, we define 
$$
\widehat{\mathscr{K}}:=\left\{G_x: \exists K \in \mathscr{K}, \,\, K\subseteq G_x\right\}.
$$
Note that $\widehat{\mathscr{K}(C)}=\mathscr{K}(C)$ if $C\subseteq X$ is  proper and, by convention, $\widehat{\{\emptyset\}}=\{\emptyset\}$. 


\begin{thm}\label{thm:uniqueness}
Let $C\subseteq X$ be a pointed cone. Then $C$ is represented by a family $\mathscr{K}$ if and only if $\widehat{\mathscr{K}}=\mathscr{K}(C)$. 
\end{thm}
\begin{proof}
First, the claim holds if $C=X$. Indeed, if $X$ is represented by $\mathscr{K}$ then each $K \in \mathscr{K}$ belongs to $\mathscr{K}_\sharp$.  
Let us suppose for the sake of contradiction that there exist $x \in X$ and a nonempty $K\in \mathscr{K}$ such that $K\subseteq G_x$. 
It follows that 
$$
X=C_{\{K\}}\subseteq C_{\{G_x\}}\subseteq X\setminus \{x\},
$$
which is impossible. Therefore $\widehat{\mathscr{K}}=\mathscr{K}(X)=\{\emptyset\}$. 
Conversely, if $\widehat{\mathscr{K}}=\{\emptyset\}$, then for each $x \in X$ and for each $K \in \mathscr{K}$ there exists a vector $y_{x,K} \in X$ such that $y_{x,K} \in K\setminus G_x$. Therefore $\langle x,y_{x,K}\rangle \ge 0$, which proves that $K \in \mathscr{K}_\sharp$, so that $X$ is represented by $\mathscr{K}$. Accordingly, suppose hereafter that $C\neq X$. 

\medskip

We are going to show that $C$ is represented by two families $\mathscr{K}_1$ and $\mathscr{K}_2$ if and only if  $\widehat{\mathscr{K}}_1=\widehat{\mathscr{K}}_2$. Thus, the conclusion will follow from Theorem \ref{thm:dualconesgeneralized}. 

\medskip

\textsc{If part.} 
Pick an element $G_a \in \widehat{\mathscr{K}}_1$, which is possible since it is nonempty. 
Then there exists $K_1 \in \mathscr{K}_1$ which is contained in $G_a$. 
Since $C$ is represented by $\mathscr{K}_1$ and $\langle a,y\rangle<0$ for all $y \in K_1$, it follows that $a\notin C$. 
Since $C$ is represented also by $\mathscr{K}_2$, there exists $K_2 \in \mathscr{K}_2$ such that $\langle a,y\rangle<0$ for all $y \in K_2$. 
Hence $K_2 \subseteq G_a$, so that $G_a \in \widehat{\mathscr{K}}_2$. This proves that $\widehat{\mathscr{K}}_1\subseteq \widehat{\mathscr{K}}_2$ and, symmetrically, $\widehat{\mathscr{K}}_2\subseteq \widehat{\mathscr{K}}_1$. 

\medskip

\textsc{Only If part.} 
Fix a vector $x\notin C_{\mathscr{K}_1}$, so that there exists $K_1 \in \mathscr{K}_1$ such that $\langle x,y\rangle <0$ for all $y \in K_1$. 
Note that $K_1\subseteq G_x$, hence by the standing assumption 
$$
G_x \in \widehat{\mathscr{K}}_1=\widehat{\mathscr{K}}_2.
$$  
By the definition of $\widehat{\mathscr{K}}_2$, there exists $K_2 \in \mathscr{K}_2$ such that $K_2 \subseteq G_x$. In particular, $\langle x,y\rangle <0$ for all $y \in K_2$, which implies that $x\notin C_{\mathscr{K}_2}$. Therefore $C_{\mathscr{K}_2}\subseteq C_{\mathscr{K}_1}$ and, symmetrically, $C_{\mathscr{K}_1}\subseteq C_{\mathscr{K}_2}$. 
This concludes the proof. 
\end{proof}


\section{Characterizations of cones with additional properties}\label{sec:additionalproperties}

As anticipated in the Introduction, we are going to characterize cones with further structural and/or topological properties, taking sometimes inspiration from the results in \cite[Section 3]{MR3957335}. 
To start with, we provide a representation of the cones $C\subseteq X$ such that $C\cup (-C)=X$: 
\begin{thm}\label{thm:propertycompleteness}
Let $C\subseteq X$ be a proper pointed cone. 
Then the following are equivalent: 
\begin{enumerate}[label={\rm (\roman{*})}]
\item \label{item1:C-CX} $C\cup (-C)=X$\emph{;}
\item \label{item2:C-CX} $H\cap K\neq \emptyset$ for all $H,K \in \mathscr{K}(C)$\emph{;}
\item \label{item3:C-CX} there exists $\mathscr{K}\subseteq \mathcal{P}(Y)$ representing $C$ such that $H\cap K\neq \emptyset$ for all $H,K \in \mathscr{K}$\emph{.}
\end{enumerate}
\end{thm}
\begin{proof}
\ref{item1:C-CX} $\implies$ \ref{item2:C-CX}. 
Note that, more explicitly, condition \ref{item2:C-CX} states that, for all $x_1,x_2\notin C$, there exists $y \in Y$ for which $\langle x_1,y\rangle<0$ and $\langle x_2,y\rangle <0$. 
Now, fix vectors $x_1,x_2 \notin C$ and observe that, by condition \ref{item1:C-CX} and Theorem \ref{thm:dualconesgeneralized}, we have that both $x_1$ and $x_2$ belong to 
$$
-C=-C_{\mathscr{K}(C)}=\left\{x \in X: \forall x_0\notin C, \exists y \in G_{x_0}, \,\, \langle x,y\rangle \le 0\right\}.
$$
In particular, there exist $y_1 \in G_{x_2}$ and $y_2 \in G_{x_1}$ such that $\langle x_1,y_1\rangle\le 0$ and $\langle x_2,y_2\rangle\le 0$. But this implies that $y:=\frac{1}{2}(y_1+y_2) \in G_{x_1} \cap G_{x_2}$. 

\medskip 

\ref{item2:C-CX} $\implies$ \ref{item3:C-CX}. Choose  $\mathscr{K}=\mathscr{K}(C)$, thanks to Theorem \ref{thm:dualconesgeneralized}.

\medskip

\ref{item3:C-CX} $\implies$ \ref{item1:C-CX}. 
Fix a vector $x \in X\setminus C$ and a family $\mathscr{K}\subseteq \mathcal{P}(Y)$ which represents $C$ such that $H\cap K\neq \emptyset$ for all $H,K \in \mathscr{K}$. 
Then there exists $K \in \mathscr{K}$ such that $\langle x,y\rangle<0$ for all $y \in K$. 
Now, fix an arbitrary $H \in \mathscr{K}$. 
By hypothesis, there exists $y \in H\cap K$. 
Thus $x$ belongs to 
$$
\left\{z \in X: \forall H \in \mathscr{K}, \exists y \in H\cap K, \,\, \langle z,y\rangle \le 0\right\}\subseteq -C,
$$
which proves that $C\cup (-C)=X$. 
\end{proof}

We proceed with the characterization of convex cones, which will allow us to provide a complete answer to the open question \cite[OP2]{MR3957335}, see Corollary \ref{cor:transitivity} below. 
To this aim, we set $\mathscr{G}:=\{G_x: x\in X\setminus \{0\}\}$. 
\begin{thm}\label{thm:convexcones}
Let $C\subseteq X$ be a pointed cone. Then the following are equivalent\textup{:}
\begin{enumerate}[label={\rm (\roman{*})}]
\item \label{item1:convexity} $C$ is convex\textup{;}
\item \label{item2:convexity} $J \notin \mathscr{K}(C)$ for all $H,K,J \in \mathscr{G}$ such that 
$H,K \notin \mathscr{K}(C)$ and $J\subseteq H\cup K$\textup{;}
\item \label{item3:convexity} there exists $\mathscr{K}\subseteq \mathcal{P}(Y)$ representing $C$ with the property that, for all $H,K,J \in \mathscr{G}$ with $J\subseteq H\cup K$, the following holds\textup{:}
$$
\mathcal{P}(H)\cap \mathscr{K}=\emptyset 
\text{ and }
\mathcal{P}(K)\cap \mathscr{K}=\emptyset 
\,\,\,\,\text{ implies }\,\,\,\,
\mathcal{P}(J)\cap \mathscr{K}=\emptyset. 
$$
\end{enumerate}
\end{thm}

Before giving the proof of Theorem \ref{thm:convexcones}, we need an auxiliary lemma.
\begin{lem}\label{lem:inclusionevilcone}
Fix nonzero vectors $a,b,c \in X$ such that $c\notin \mathrm{co}(\mathrm{cone}(\{a,b\}))$. Then 
$$
G_c\setminus (G_a\cup G_b)\neq \emptyset.
$$
\end{lem}
\begin{proof}
We divide the proof into three subcases.


\medskip

\textsc{Case (i): $b,c\in \mathrm{span}(\{a\})$}. Suppose that $b=\beta a$ and $c=\gamma a$, for some nonzero $\beta,\gamma \in \mathbf{R}$. Since $c\in \mathrm{span}(\{a\})\setminus \mathrm{co}(\mathrm{cone}(\{a,b\}))$ then necessarily $\gamma<0<\beta$. Then, it is enough to pick an arbitrary $y \in Y$ with the property that $\langle a,y\rangle=1$. Note that this is possible by an application of Hahn--Banach theorem (in the following cases we proceed similarly). It follows by construction that $\langle b,y\rangle=\beta$ and $\langle c,y\rangle=\gamma$. Therefore $y \in G_c\setminus (G_a\cup G_b)$.

\medskip

\textsc{Case (ii): $c \in \mathrm{span}(\{a,b\})$ and $b\notin \mathrm{span}(\{a\})$}. Suppose that $c=\alpha a+\beta b$, for some $\alpha,\beta \in \mathbf{R}$, and $b$ is not a multiple of $a$. Since $c\neq 0$ and $c\notin \mathrm{co}(\mathrm{cone}(\{a,b\}))$ then $\min\{\alpha,\beta\}<0$. Suppose, up to relabelings, that $\alpha<0$. If $\beta\le 0$, pick $y \in Y$ such that $\langle a,y\rangle=\langle b,y\rangle=1$, hence $\langle c,y\rangle=\alpha+\beta\le \alpha<0$. In the opposite, if $\beta>0$, pick $y \in Y$ such that $\langle a,y\rangle=1$ and $\langle b,y\rangle=-\alpha/(2\beta)$, hence $\langle c,y\rangle=\alpha/2<0$. In both cases, the conclusion follows as above.

\medskip

\textsc{Case (iii): $c \notin \mathrm{span}(\{a,b\})$}. It enough to pick $y \in Y$ such that $\langle a,y\rangle=\langle b,y\rangle=0$ and $\langle c,y\rangle=-1$. The conclusion follows as above.
\end{proof}

\begin{proof}[Proof of Theorem \ref{thm:convexcones}]
\ref{item1:convexity} $\implies$ \ref{item2:convexity}. First of all, note that the implication holds if $C=\{0\}$, hence let us assume hereafter that $C\neq \{0\}$ or equivalently $\mathscr{K}(C)\neq \mathscr{G}$. 
Fix $H,K \in \mathscr{G}\setminus \mathscr{K}(C)$ and $J \in \mathscr{G}$ such that $J\subseteq H\cup K$. 
By the definition of $\mathscr{G}$, there exist nonzero vectors $a,b,c \in X$ such that $H=G_a$, $K=G_b$, and $J=G_c$. 
It follows by Theorem \ref{thm:dualconesgeneralized} that $a,b \in C$. In addition, thanks to Lemma \ref{lem:inclusionevilcone}, we have $c\in \mathrm{co}(\mathrm{cone}(\{a,b\}))$. Since $C$ is a convex cone, we conclude that $c \in C$, that is, $G_c=J\notin \mathscr{K}(C)$.

\medskip

\ref{item2:convexity} $\implies$ \ref{item1:convexity}. Fix $a,b \in C$ and $t \in (0,1)$ and define $c:=ta+(1-t)b$. Then by definition of $\mathscr{K}(C)$, we have $G_a,G_b\notin \mathscr{K}(C)$. Pick $y \in G_c$, so that $\langle c,y\rangle=t \langle a,y\rangle+(1-t)\langle b,y\rangle <0$. This implies that $\min\{\langle a,y\rangle,\langle b,y\rangle \}<0$. Therefore $y \in G_a\cup G_b$. By \ref{item2:convexity}, we have $G_c\notin \mathscr{K}(C)$, which implies that $c\in C$.

\medskip

\ref{item2:convexity} $\implies$ \ref{item3:convexity}. It follows by setting $\mathscr{K}=\mathscr{K}(C)$ and using Theorem \ref{thm:dualconesgeneralized}.

\medskip

\ref{item3:convexity} $\implies$ \ref{item2:convexity}. It follows by Theorem \ref{thm:uniqueness}.
\end{proof}

\begin{defi}
Let $S\subseteq X$ be a nonempty subset. We denote its \emph{dual cone} by 
$$
S^\prime:=\left\{y \in Y: \forall x \in S, \,\, \langle x,y\rangle \ge 0\right\}.
$$
\end{defi}
There will be no risk of confusion with the topological dual of a tvs.  
Note that $S^\prime$ is a closed convex cone in $Y$, that is, $S^\prime \in \mathscr{C}(Y)$. Also, $S^\prime=(\overline{\mathrm{co}}(\mathrm{cone}(S)))^\prime$.




\begin{lem}\label{lem:standardlemmaduality}
Fix cones $C,D\in \mathscr{C}(X)$. Then $C\subseteq D$ if and only if $D^\prime \subseteq C^\prime$. 
\end{lem}
\begin{proof}
It is clear that, $D^\prime \subseteq C^\prime$ whenever $C\subseteq D$.  
Conversely, suppose that $C$ is not contained in $D$, so that there exists a vector $x_0 \in C\setminus D$. 
It follows by the Strong Separating Hyperplane Theorem, 
see e.g. \cite[Theorem 8.17]{MR2317344}, 
that there exist a functional $y_0 \in Y$ and a real $\alpha \in \mathbf{R}$ such that 
$$
\forall x \in D, \quad 
\langle x_0,y_0\rangle <\alpha \le \langle x,y_0\rangle. 
$$
Since $D$ is a cone, then $\alpha \le 0$. This implies that $y_0 \in D^\prime \setminus C^\prime$. 
\end{proof}

The following corollary is immediate from Lemma \ref{lem:standardlemmaduality}:
\begin{cor}\label{cor:uniquenessconedualclosedconvex}
The map $\mathscr{C}(X)\to \mathscr{C}(Y)$ defined by $C\mapsto C^\prime$ is bijective. 
\end{cor}

Another application of Lemma \ref{lem:standardlemmaduality} follows below, providing an abstract version of a result of Evren \cite[Theorem 5 and Theorem 5a]{MR2398816}:
\begin{cor}
Let $C\subseteq X$ and $A,B\subseteq Y$ be convex cones such that both $A$ and $B$ are relatively closed and contained in the topological dual $D^\prime$, where $D:=C-C$. 
Then $A\subseteq B$ if and only if $B^\prime \cap D \subseteq A^\prime \cap D$. 
\end{cor}
\begin{proof}
Note that $D$ is the vector space generated by $C$. Setting $X=D^\prime$ in Lemma \ref{lem:standardlemmaduality} and using the duality of the pair $(D,D^\prime)$, we obtain that $A\subseteq B$ if and only if the dual cone of $B$ with respect to $D$, which coincides with $B^\prime \cap D$, is contained in the dual cone of $A$ with respect to $D$, which is $A^\prime \cap D$, concluding the proof.
\end{proof}

In the next result, we characterize the class of certain cones which contains the dual of a given closed convex cone $H\neq \{0\}$: 
\begin{thm}\label{thm:includeclosedconvexconefunctions}
Let $C\subseteq X$ be a proper pointed cone 
which is represented by a family $\mathscr{K}\subseteq \mathcal{P}(Y)$ of nonempty convex compact sets. 
Fix also a closed convex cone $H\subseteq Y$ with $H\neq \{0\}$. 
Then the following are equivalent\emph{:}
\begin{enumerate}[label={\rm (\roman{*})}]
\item \label{item1:H-HX} $H^\prime \subseteq C$\emph{;}
\item \label{item2:H-HX} $H\cap K\neq \emptyset$ for all $K \in \mathscr{K}(C)$\emph{;}
\item \label{item3:H-HX} there exists $\mathscr{K}\subseteq \mathcal{P}(Y)$ representing $C$ such that $H\cap K\neq \emptyset$ for all $K \in \mathscr{K}$. 
\end{enumerate}
\end{thm}
\begin{proof}
\ref{item1:H-HX} $\implies$ \ref{item2:H-HX}. 
Let us suppose for the sake of contradiction that $H^\prime \subseteq C$ and that there exists $K \in \mathscr{K}(C)$ such that $H \cap K=\emptyset$. Recalling that $Y$ is locally convex and $Y^\prime=X$, it follows by the Strong Separating Hyperplane Theorem, 
see e.g. \cite[Theorem 8.17]{MR2317344}, 
that there exist $x_0 \in X$ and $\alpha,\beta \in \mathbf{R}$ such that 
$$
\forall h \in H, \forall k \in K, \quad 
\langle x_0,k\rangle \le  \alpha < \beta \le \langle x_0,h\rangle. 
$$
However, since $H$ is a cone containing a nonzero vector, then $\beta \le 0$. On the one hand, it follows that $\langle x_0,k\rangle<0$ for all $k \in K \in \mathscr{K}(C)$, hence $x_0\notin C_{\mathscr{K}(C)}=C$ by Theorem \ref{thm:dualconesgeneralized}. 
On the other hand, we have that $\langle x_0,h\rangle \ge 0$ for all $h \in H$. Therefore $x_0 \in H^\prime\subseteq C$, providing the desired contradiction. 

\medskip

\ref{item2:H-HX} $\implies$ \ref{item3:H-HX}. Choose  $\mathscr{K}=\mathscr{K}(C)$, thanks to Theorem \ref{thm:dualconesgeneralized}.

\medskip 

\ref{item3:H-HX} $\implies$ \ref{item1:H-HX}. 
Pick a vector $x \in X$ such that $\langle x,y \rangle \ge 0$ for all $y \in H$. Since $C$ is represented by $\mathscr{K}$ and $H\cap K\neq \emptyset$ for all $K \in \mathscr{K}$, it follows that $x \in C$. Therefore the dual cone $H^\prime$ is contained in $C$. 
\end{proof}

\begin{rmk}
As it turns out from Theorem \ref{thm:closedcone} and Remark \ref{rmk:closedconvexforclosedcones_Banach} below, if $X$ is a Banach space then the class of cones $C\subseteq X$ satisfying the hypothesis of Theorem \ref{thm:includeclosedconvexconefunctions} is precisely the family of closed cones $C\subseteq X$. 
\end{rmk}

Now, we characterize closed convex cones $C\subseteq X$. 
While it is an easy consequence of the Bipolar theorem that $C=C^{\prime\prime}$, we show also that, if $C=D^\prime$ for some closed convex cone $D\subseteq Y$, then necessarily $D=C^\prime$. 
In addition, we characterize the interior of such cones. 
This provides a generalization of the main results contained in \cite{Dubra_et_al, MR3182925}, cf. also \cite{MR3030771, MR2291506, MR2977438}. 

To this aim, recall that a locally convex topology $\tau$ on $X$ is said to be \emph{consistent} if the topological dual of $(X,\tau)$ coincides with $X^\prime$. 
\begin{thm}\label{thm:dualdualcone}
Fix cones $C \in \mathscr{C}(X)$ and $D \in \mathscr{C}(Y)$. 
Then the following are equivalent\emph{:}
\begin{enumerate}[label={\rm (\roman{*})}]
\item \label{item01:base} $D=C^\prime$\emph{;} 
\item \label{item02:base} $C=D^\prime$\emph{.}
\end{enumerate}

If, in addition, $C$ and $D^\prime$ have nonempty interior with respect to a consistent locally convex topology $\tau$ on $X$, then they are also equivalent to\emph{:} 
\begin{enumerate}[label={\rm (\roman{*})}]
\setcounter{enumi}{2}
\item \label{item03:base} $\mathrm{int}(C)=\{x \in X: \forall y \in D\setminus \{0\}, \,\, \langle x,y\rangle > 0\}$. 
\end{enumerate}
\end{thm}

The proof of  \ref{item01:base} $\implies$ \ref{item03:base} can be found in \cite[Lemma 2.17]{MR2317344}, and the implication \ref{item01:base} $\implies$ \ref{item02:base}, which 
can be rephrased as \textquotedblleft every closed convex cone $C\subseteq X$ is represented by the family $\{\{y\}: y \in C^\prime\}$,\textquotedblright\ corresponds to the known fact
\begin{equation}\label{eq:CCprimeprime}
C=\{x\in X: \forall y \in C^\prime, \,\, \langle x,y\rangle \ge 0\}.
\end{equation}
We include their proofs for the sake of completeness. 
Before we proceed to the proof of Theorem \ref{thm:dualdualcone}, we recall the following lemma:
\begin{lem}\label{lem:convextvsclusreinterior}
Let $S$ be a convex set with nonempty interior in a tvs. Then $\mathrm{cl}(\mathrm{int}(S))=\mathrm{cl}(S)$ and $\mathrm{int}(\mathrm{cl}(S))=\mathrm{int}(S)$.
\end{lem}
\begin{proof}
See \cite[Lemma 5.28]{MR2378491}. 
\end{proof}

\begin{proof}
[Proof of Theorem \ref{thm:dualdualcone}]
\ref{item01:base} $\implies$ \ref{item02:base}. 
Fix a vector $y_0$ in the one-sided polar 
$$
C^\bullet:=\{y \in Y: \forall x \in C,\,\, \langle x,y\rangle \le 1\}.
$$
Since $C$ is a cone, then $\langle \lambda x,y_0\rangle \le 1$ for all $x \in C$ and real $\lambda > 0$. 
Therefore $\langle x,y_0\rangle \le 0$ for all $x \in C$, hence $C^\bullet=-C^\prime$ and, thanks to the Bipolar theorem, see e.g. \cite[Theorem 8.12]{MR2317344},  
we conclude that 
$$
C=C^{\bullet\bullet}=C^{\prime\prime}=D^\prime.
$$
%
%

\medskip

\ref{item02:base} $\implies$ \ref{item01:base}. 
Suppose that $C=D^\prime$. By the previous point, we know that $C=C^{\prime\prime}$. Therefore $D=C^\prime$ by Corollary \ref{cor:uniquenessconedualclosedconvex}. 

\medskip

Hereafter, suppose that $C$ has nonempty interior with respect to $\tau$. 

\medskip

\ref{item01:base} $\implies$ \ref{item03:base}. 
Pick a nonzero vector $x \in \mathrm{int}(C)$ and a functional $y \in C^\prime$. Since $x$ is an interior point, there exists a balanced neighborhood $U$
 of $0$ such that $x+U \subseteq \mathrm{int}(C)$. 
By the definition of the dual cone, we have $\langle x+u,y\rangle \ge 0$ for all $u \in U$. 
Since $U=-U$, it follows that $\langle u,y\rangle=0$ for all $u \in U$, which implies that $y=0$. This proves that 
$$
\mathrm{int}(C)\subseteq \{x \in X: \forall y \in C^\prime\setminus \{0\}, \,\, \langle x,y\rangle >0\}. 
$$

Conversely, pick a nonzero vector $x_0$ such that $\langle x_0,y\rangle>0$ for all nonzero $y \in C^\prime$. 
If $x_0\notin \mathrm{int}(C)$, it would follow by the Interior Separating Hyperplane Theorem, see e.g. \cite[Theorem 8.16]{MR2317344}, that there exist a nonzero functional $y_0 \in Y$ and a real $\alpha \in \mathbf{R}$ such that 
$$
\forall x \in \mathrm{cl}(\mathrm{int}(C)), \quad 
\langle x_0,y_0\rangle \le \alpha \le \langle x,y_0\rangle. 
$$
Since $\mathrm{cl}(\mathrm{int}(C))=C$, by Lemma \ref{lem:convextvsclusreinterior}, and $C$ is a cone, we obtain that $\alpha \le 0$. 
Hence $y_0$ belongs to $C^\prime$ and $\langle x_0,y_0\rangle \le 0$. However, this contradicts the standing assumption $\langle x_0,y\rangle>0$ for all nonzero $y \in C^\prime$, which shows the opposite inclusion. 

\medskip

\ref{item03:base} $\implies$ \ref{item02:base}. 
Set $E:=D^\prime$. 
Thanks to Lemma \ref{lem:convextvsclusreinterior}, the equivalence \ref{item01:base} $\Longleftrightarrow$ \ref{item02:base}, and the implication \ref{item01:base} $\implies$ \ref{item03:base}, we obtain that 
\begin{displaymath}
\begin{split}
C
=\textnormal{cl}\left(\textnormal{int}\left(C\right)\right)
&=\textnormal{cl}\left(\left\lbrace x\in X:\forall y\in D\setminus \{0\},\,\, \langle x,y \rangle >0 \right\rbrace\right)   \\ 
&=\textnormal{cl}\left(\left\lbrace x\in X:\forall y\in E^\prime \setminus \{0\},\,\, \langle x,y \rangle >0 \right\rbrace\right)   \\
&=\textnormal{cl}(\mathrm{int}(E))=E=D^\prime,   \\
\end{split}
\end{displaymath}
which completes the proof. 
\end{proof}

In particular, Theorem \ref{thm:dualdualcone} shows that the interior of $C$ is the same for all consistent locally convex topologies $\tau$ on $X$. This is analogous to Mackey's theorem, which states that all consistent topologies have the same closed convex sets, see e.g. \cite[Theorem 8.9]{MR2317344}. 

At this point, we characterize (proper) open convex cones:
\begin{prop}\label{prop:openconvex}
Suppose that $X$ is a normed space and let $C\subseteq X$ be a proper subset. Then $C$ is a norm-open convex cone if and only if there exists a nonempty 
$\sigma(Y,X)$-compact $U\subseteq Y$ such that 
\begin{equation}\label{eq:openconvexrepresentation}
C=\{x\in X:\forall y \in U, \,\, \langle x,y\rangle>0\}.
\end{equation}

In addition, if $V\subseteq Y$ is another nonempty $\sigma(Y,X)$-compact set with the same property, then 
$
\overline{\mathrm{co}}(\mathrm{cone}(U))=\overline{\mathrm{co}}(\mathrm{cone}(V)).
$
\end{prop}

In the proof of Proposition \ref{prop:openconvex}, we will need the following standard lemma. 

\begin{lem}\label{lem:intersectionopen}
Let $X$ be a normed space and fix a nonempty $\sigma(X,Y)$-compact set $U\subseteq X$. 
Then $\bigcap_{x \in U}G_x$ is norm-open.
\end{lem}
\begin{proof}
Set $G:=\bigcap_{x \in U}G_x$ and suppose that $G\neq \emptyset$, otherwise there is nothing to prove. 
Since each section $\langle \cdot,y \rangle$ is $\sigma(X,Y)$-continuous, by Weierstrass' theorem it is possible to define the function $f: Y\to \mathbf{R}$ by 
$$
\forall y \in Y, \quad 
f(y):=\max_{x \in U}\,\langle x,y\rangle.
$$
Note that $f(y)<0$ if and only if $y\in G$. 
Now, suppose that a sequence $\left(y_n\right)_{n\in \mathbf{N}}$ is norm-convergent to $0 \in Y$. 
Since $f(0)=0$, we have 
$$
\forall n \in \mathbf{N}, \quad 
\left|f(y_n)\right|\leq 
\max\limits_{x\in U}\left|\langle x,y_n\rangle\right|\leq 
\lVert y_n \rVert \cdot \max\limits_{x\in U}\,\lVert x \rVert,
$$
hence $\lim_n f(y_n)=0$, i.e., $f$ is norm-continuous at $0$. 
Thus, since $f$ is sublinear, it must be norm-continuous. 
Therefore $G=f^{-1}\left((-\infty,0)\right)$ is norm-open.
\end{proof}

\begin{proof}
[Proof of Proposition \ref{prop:openconvex}]
\textsc{If part.} 
Let $U$ be a nonempty $\sigma(Y,X)$-compact $U\subseteq Y$ such that \eqref{eq:openconvexrepresentation} holds. It is straightforward to see that $C$ is a convex cone. Moreover, it is open in the norm topology 
by Lemma \ref{lem:intersectionopen}.

\medskip

\textsc{Only If part.} 
Thanks to Lemma \ref{lem:convextvsclusreinterior}, we have $C=\mathrm{int}(\mathrm{cl}(C))$ in the norm topology. 
Applying the equivalence \ref{item01:base} $\Longleftrightarrow$ \ref{item03:base} in Theorem \ref{thm:dualdualcone} to the cone $\mathrm{cl}(C) \in \mathscr{C}(X)$, we obtain that 
$$
C=\left\{x\in X: \forall y \in (\mathrm{cl}(C))^\prime\setminus \{0\}, \,\, \langle x,y\rangle>0\right\}.
$$
At this point, let $S_Y:=\{y \in Y: \|y\|=1\}$ be the unit sphere of the topological dual of $X$, which is $\sigma(Y,X)$-compact by Alaoglu--Bourbaki's theorem, see e.g. \cite[Theorem 8.8]{MR2317344}. Hence \eqref{eq:openconvexrepresentation} holds with the $\sigma(Y,X)$-compact set $U:=S_Y\cap (\mathrm{cl}(C))^\prime$. 

\medskip

For the second part, let $U$ and $V$ be nonempty $\sigma(Y,X)$-compact set which represent the norm-open convex cone $C$ as in \eqref{eq:openconvexrepresentation}. Fix also $x_0 \in C$ and a vector $x \in X$ such that $\langle x,y\rangle \ge 0$ for all $y \in U$. Then $\alpha x+(1-\alpha)x_0 \in C$ for all $\alpha \in [0,1)$. Since both $U$ and $V$ represent $C$, it follows that $\langle \alpha x+(1-\alpha)x_0,y\rangle >0$ for all $\alpha \in [0,1)$. By the continuity of the duality map and taking the limit $\alpha \to 1^-$, we obtain that $\langle x,y\rangle \ge 0$ for all $y \in V$. This implies that $U^\prime \subseteq V^\prime$. With a symmetric argument, we have $V^\prime\subseteq U^\prime$. Therefore 
$$
(\overline{\mathrm{co}}(\mathrm{cone}(U)))^\prime=U^\prime=V^\prime=(\overline{\mathrm{co}}(\mathrm{cone}(V)))^\prime. 
$$
The conclusion follows by Corollary \ref{cor:uniquenessconedualclosedconvex}. 
In particular, $\overline{\mathrm{co}}(\mathrm{cone}(U))$ has to be equal to $(\mathrm{cl}(C))^\prime$.
\end{proof}

Finally, everything is ready to provide a representation of arbitrary closed cones in Banach spaces: 
\begin{thm}\label{thm:closedcone}
Suppose that $X$ is a Banach space and let $C\subseteq X$ be a subset. Then the following are equivalent\emph{:}
\begin{enumerate}[label={\rm (\roman{*})}]
\item \label{item:closed1} $C$ is a norm-closed cone\emph{;} 
\item \label{item:closed2} $C$ is represented by a family $\mathscr{K}$ of nonempty $\sigma(Y,X)$-compact sets\emph{;}
\item \label{item:closed3} $C$ is represented by a family $\mathscr{K}$ of nonempty $\sigma(Y,X)$-compact convex sets\emph{.}
\end{enumerate}
\end{thm}

\begin{proof}
The claim holds if $C=X$ by choosing, e.g., $\mathscr{K}=\{\{0\}\}$, hence suppose hereafter that $C$ is proper. Also, let $\tau$ be the norm topology on $X$

\medskip 

\ref{item:closed1} $\implies$ \ref{item:closed2}. 
Set $D:=X\setminus C$, which is a proper $\tau$-open cone. 
Since $\tau$ is locally convex, for each $x \in D$ there exists a $\tau$-open neighborhood $V_x$ of zero such that $D_x:=\mathrm{cone}(x+V_x)$ is a proper $\tau$-open convex cone contained in $D$. 
In addition, 
$$
D=\bigcup\nolimits_{x \in D}D_x.
$$ 
At this point, it follows by Proposition \ref{prop:openconvex} that, for each $x \in D$, there exists a nonempty $\sigma(Y,X)$-compact set $U_x\subseteq Y$ such that $D_x$ can be written as $\{a \in X: \forall y \in U_x,\,\, \langle a,y\rangle >0\}$.
With these premises, we conclude that
\begin{displaymath}
\begin{split}
    C&=\bigcap\nolimits_{x \in D}(X\setminus D_x)\\
     &=\{a \in X: \forall x \in D, \,\, a\notin D_x\}\\
     &=\{a \in X: \forall x \in D, \exists y \in U_x, \,\,\langle a,y\rangle \le 0\}.
\end{split}
\end{displaymath}
Therefore $C$ is represented by the family $\mathscr{K}:=\{-U_x: x\in D\}$. 

\medskip

\ref{item:closed2} $\implies$ \ref{item:closed3}. 
It is enough to replace each $U_x$ in the above proof with $K_x:=\mathrm{co}(U_x)$ for each $x \in D$, which is $\sigma(Y,X)$-compact for each $x \in D$, thanks to the Krein--\v{S}mulian theorem, see e.g. \cite[Theorem 8.28]{MR2317344}. For, we show that 
\begin{equation}\label{eq:openconvexconeKx}
\forall x \in D, \quad 
D_x=\{a \in X: \forall y \in K_x,\,\, \langle a,y\rangle >0\}
\end{equation}
and repeat verbatim the same argument. 
By the $\sigma(Y,X)$-compactness of $U_x$ and the $\sigma(Y,X)$-continuity of the sections $\langle x,\cdot\rangle$, we can define the positive reals 
$\varepsilon_{x,a}:=\min\nolimits_{y \in U_x}\langle a,y\rangle$ for each $x \in D$ and $a \in D_x$. 
Hence, passing to the closed convex hull, $\langle a,y\rangle \ge \varepsilon_{x,a}$ for all $x \in D$, $a \in D_x$, and $y \in K_x$, implying \eqref{eq:openconvexconeKx}. 

\medskip

\ref{item:closed3} $\implies$ \ref{item:closed2}. This is clear.

\medskip

\ref{item:closed2} $\implies$ \ref{item:closed1}. 
Let $\mathscr{K}$ be a family of nonempty $\sigma(Y,X)$-compact sets which represents $C$. 
%
Pick $K \in \mathscr{K}$ and a sequence $(x_n)_{n \in \mathbf{N}}$ in $C$ which is $\tau$-convergent to some $x \in X$. We claim that $x \in C$. 
For, since $C$ is represented by $\mathscr{K}$, there exists a sequence $(y_n)_{n \in \mathbf{N}}$ in $K$ such that $\langle x_n,y_n\rangle \ge 0$ for all $n \in \mathbf{N}$. 
By the hypothesis that $K$ is $\sigma(Y,X)$-compact, there exists a subsequence $(y_{n_k})_{k \in \mathbf{N}}$ which is $\sigma(Y,X)$-convergent to some $y \in K$. 
Since $K$ is necessarily norm-bounded by \cite[Corollary 2.6.9]{MegginsonBook}, the restriction of the duality map $\langle \cdot, \cdot\rangle$ to $X\times K$ is jointly $\tau\times \sigma(Y,X)$-continuous \cite[Corollary 6.40]{MR2378491}. 
It follows that the subsequence $\left(\langle x_{n_k},y_{n_k}\rangle\right)_{k \in \mathbf{N}}$ converges to $\langle x,y\rangle$. Therefore $\langle x,y\rangle\geq 0$ and, by the arbitrariness of $K$, we get $x \in C$. 
\end{proof}

\begin{rmk}\label{rmk:closedconvexforclosedcones_Banach}
Note that implication \ref{item:closed1} $\implies$ \ref{item:closed2} holds also if $X$ is not necessarily complete. On the other hand, our proof of \ref{item:closed2} $\implies$ \ref{item:closed1} relies on the completeness assumption in applying \cite[Corollary 2.6.9]{MegginsonBook}. Indeed, there exist normed spaces $X$ which admit norm-unbounded $\sigma(Y,X)$-compact subsets of $X^\prime$. 
\end{rmk}


The following corollary provides an abstract representation for certain cones which represent \emph{justifiable preferences}, as defined by Lehrer and Teper in \cite{MR2888839}. Accordingly, it provides a generalization of \cite[Theorem 1]{MR2888839} for Anscombe--Aumann acts and \cite[Theorem 5]{MR3957335} for lotteries. In a sense, it complements the known representation for closed convex cones in \eqref{eq:CCprimeprime}:
\begin{cor}\label{cor:generatingclosedconvexinterior}
Suppose that $X$ is a normed space and let $C\subseteq X$ be a set. 

Then $C$ is a norm-closed cone such that $C\cup (-C)=X$ and $C\setminus (-C)$ is convex if and only if there exists a nonempty $\sigma(Y,X)$-compact set $K\subseteq Y$ for which $C$ is represented by $\{K\}$, i.e., 
\begin{equation}\label{eq:representationexists}
C=\{x \in X: \exists y \in K, \,\, \langle x,y\rangle \ge 0\}.
\end{equation}
\end{cor}
\begin{proof}
First, note that the claim holds if $C=X$ by choosing $K=\{0\}$. Hence, suppose hereafter that $C$ is a proper subset of $X$ and define $D:=C\setminus (-C)$. Note that $D$ is a cone. 

\medskip 

\textsc{If part.} It follows by Lemma \ref{lem:intersectionopen} and duality that $X\setminus C=\bigcap\nolimits_{y \in K}\{x \in X: \langle x,y\rangle <0\}$ is open in the norm topology. Hence $C$ is a norm-closed cone. It is also clear that, if representation \eqref{eq:representationexists} holds, then $\{x,-x\} \cap C\neq \emptyset$ for all $x \in X$, so that $C\cup (-C)=X$. 
Finally, pick $x_1,x_2 \in D$ and $\alpha \in (0,1)$. 
Since $x_1,x_2\notin -C$ and $-C$ is represented by $\{-K\}$ then $\langle x_1,y\rangle >0$ and $\langle x_2,y\rangle >0$ for all $y \in K$. 
Hence $\langle \alpha x_1+(1-\alpha)x_2, y\rangle>0$ for all $y \in K$, which implies that $D$ is convex. 

\medskip

\textsc{Only If part.} Since $C\neq X$, then $D$ is a proper convex cone. In addition, $D=(C\cup (-C))\setminus (-C)=X\cap (-C)^c$ is norm-open. 
Hence it follows by Proposition \ref{prop:openconvex} that there exists a $\sigma(Y,X)$-compact set $U\subseteq Y$ which satisfies 
$$
D=\{x \in X: \forall y \in U, \,\, \langle x,y\rangle >0\}.
$$
Equivalently, $-D=\{x \in X: \forall y \in U, \,\, \langle x,y\rangle <0\}$. 
Note that $\{C,-D\}$ is a partition of $X$. 
Therefore
\begin{displaymath}
\begin{split}
    C=X\setminus (-D)=\{x \in X: \exists y \in U,\,\,\langle x,y\rangle \ge 0\},
\end{split}
\end{displaymath}
completing the proof. 
\end{proof}


\section{Applications in Decision Theory}\label{sec:appl}

The principle underlying the representation results used in quoted results of decision theory is simple: Let $\mathcal{R}$ be a binary relation on a given set $S$,\footnote{As usual, the interpretation is that a binary relation $\mathcal{R}\subseteq S^2$ represents a preference $\succeq$ on $S$ defined by $a\succeq b$ if and only if $(a,b) \in \mathcal{R}$ for all $a,b \in S$.} and suppose also that there exist a topological vector space $X$, a pointed cone $C\subseteq X$, and a function $f: S^2\to X$ such that
\begin{equation}\label{eq:aumanncone}
\forall a,b \in S, \quad \quad (a,b) \in \mathcal{R} \,\,\Longleftrightarrow \,\,
f(a,b) \in C.
\end{equation}
Then, it is enough to represent the cone $C$: Indeed, it follows by Theorem \ref{thm:dualconesgeneralized} that there exists a nonempty family $\mathscr{K}\subseteq \mathcal{P}(Y)$ of nonempty subsets such that 
$$
\forall a,b \in S, \quad \quad (a,b) \in \mathcal{R} \,\,\Longleftrightarrow \,\,
\left(\forall K \in \mathscr{K}, \exists y \in K,\, \langle f(a,b),y\rangle \ge 0\right).
$$
Also, if the cone $C$ can be chosen with additional properties (depending on the assumptions on the binary relation $\mathcal{R}$), then Theorem \ref{thm:dualconesgeneralized} can be replaced with one of its refinements given in Section \ref{sec:additionalproperties}. 

\begin{example}\label{example:aumanncone}
Suppose that $S$ is a convex subset of a vector space and set $X:=\mathrm{span}(S)$. Define also the function $f:S^2\to X$ by $f(a,b):=a-b$. 
Then it is easy to check that a reflexive or irreflexive binary relation $\mathcal{R}$ on $S$ satisfies \eqref{eq:aumanncone} for some cone $C\subseteq X$ provided that $\mathcal{R}$ satisfies the so-called \emph{independence axiom}, i.e., 
\begin{equation}\label{eq:stronindependence}
(a,b) \in \mathcal{R} 
\,\,\Longleftrightarrow \,\,
(\alpha a+(1-\alpha)c,\alpha b+(1-\alpha)c) \in \mathcal{R}
\end{equation}
for all $a,b,c \in S$ and all $\alpha \in (0,1)$. In addition, $C$ is pointed if and only if $\mathcal{R}$ is reflexive. In such case, the cone $C=C_{\mathcal{R}}$ defined by
$$
C_{\mathcal{R}}:=\{\alpha (a-b): \alpha \ge 0, (a,b) \in \mathcal{R}\}
$$
is the \emph{Aumann's cone} generated by $\mathcal{R}$, cf. \cite{MR174381} and \cite[Footnote 15]{Dubra_et_al}. It is worth to remark, following \cite{Dubra_et_al}, that if $C_{\mathcal{R}}$ is closed and convex then the binary relation $\mathcal{R}$ must be transitive and continuous.\footnote{Recall that a binary relation $\mathcal{R}$ on a set $S$ is said to be continuous if, for all nets $\left(x_{\alpha},y_{\alpha}\right)_{\alpha\in A}\in \mathcal{R}^A$ with $\lim_{\alpha \in A}x_{\alpha}=x\in S$ and $\lim_{\alpha \in A}y_{\alpha}=y \in S$, it follows that $(x,y)\in \mathcal{R}$.}
\end{example}

As an immediate application, e.g., we recover \cite[Theorem 1]{MR3957335}:
\begin{cor}\label{cor:firstresultok}
Let $\mathcal{R}$ be a binary relation on the set $\Delta(Z)$ of Borel probability measures on a separable metric space $Z$. 

Then $\mathcal{R}$ is reflexive and satisfies the independence axiom \eqref{eq:stronindependence} if and only if there exists a nonempty family $\mathscr{U}\subseteq \mathcal{P}(\mathrm{C}_b(Z))$ of nonempty subsets such that
\begin{equation}\label{eq:representationdecisiontheory}
(p,q) \in \mathcal{R}\,\,\,\Longleftrightarrow\,\,\,
\forall U \in \mathscr{U}, \exists u \in U, \
\mathbf{E}_p[u] \ge \mathbf{E}_q[u],
\end{equation}
for all $p,q \in \Delta(Z)$, 
where $\mathbf{E}_r[u]:=\int u\,\mathrm{d}r$ for all $r \in \Delta(Z)$ and $u \in \mathrm{C}_b(Z)$. 
\end{cor}
\begin{proof}
The \textsc{If} part is easy to check. 
For the \textsc{Only If part}, 
set $S:=\Delta(Z)$ and let $X:=\mathrm{ca}(Z)=\mathrm{span}(S)$ be the set of countably additive finite signed Borel measures on $Z$. 
It follows by the Affine Representation Lemma, see e.g. \cite[p. 965]{MR3957335}, that $(X,\mathrm{C}_b(Z))$ is a dual pair, with 
$\langle \mu,u\rangle:=\int u\,\mathrm{d}\mu$ for all $\mu \in X$ and $u \in \mathrm{C}_b(Z)$. 
By Example \ref{example:aumanncone}, $\mathcal{R}$ is represented by the pointed cone $C_{\mathcal{R}}\subseteq X$. 
The claim follows by Theorem \ref{thm:dualconesgeneralized} and the linearity of the duality map. 
\end{proof}

As anticipated in Section \ref{sec:additionalproperties}, we obtain a complete characterization of preorders which satisfies the independence axiom, answering the open question \cite[OP2]{MR3957335}.

\begin{cor}\label{cor:transitivity}
Let $\mathcal{R}$ be a binary relation on the set $\Delta(Z)$ of Borel probability measures on a separable metric space $Z$, and define
$$
\mathscr{U}_0:=\left\{\left\{u \in \mathrm{C}_b(Z): \mathbf{E}_p[u]<\mathbf{E}_q[u]\right\}: p,q \in \Delta(Z)\right\}.
$$

Then $\mathcal{R}$ is reflexive, transitive, and satisfies the independence axiom \eqref{eq:stronindependence} if and only if there exists a nonempty family $\mathscr{U}\subseteq \mathcal{P}(\mathrm{C}_b(Z))$ of nonempty subsets such that \eqref{eq:representationdecisiontheory} holds 
for all $p,q \in \Delta(Z)$ and 
$$
\mathcal{P}(U_1)\cap \mathscr{U}=\emptyset 
\text{ and }
\mathcal{P}(U_2)\cap \mathscr{U}=\emptyset 
\,\,\,\,\text{ implies }\,\,\,\,
\mathcal{P}(V)\cap \mathscr{U}=\emptyset
$$
for all $U_1,U_2,V \in \mathscr{U}_0$ with $V\subseteq U_1\cup U_2$.
\end{cor}
\begin{proof}
The proof goes analogously as in Corollary \ref{cor:firstresultok}, taking into account Theorem \ref{thm:convexcones} and that $C_{\mathcal{R}}$ is convex by \cite[Lemma A.2(d)]{MR3957335}.
\end{proof}

Several additional characterizations in \cite{MR3957335} can be obtained analogously. 
In this work, the 
advantage is that our results can be applied to different settings. 
For instance, in \cite[OP4]{MR3957335} the authors leave as open question to extend their analysis \textquotedblleft to the context of uncertainty, where one would instead take the Anscombe–Aumann expected utility theorem as the starting point. [...] Nothing is known at present about the structure of affine and monotonic (and state independent) preferences over (finite) acts (on a finite state space) that fail to satisfy continuity, completeness, and/or transitivity.\textquotedblright\ 

A first attempt to solve this question follows below. 
%
Following e.g. \cite{MR3030771, MR2291506, MR2977438}, 
given a nonempty finite set $\Omega$ and a separable metric space $Z$, we denote by 
$$
\mathscr{F}(\Omega,Z):=\Delta(Z)^\Omega
$$
the set of functions $f: \Omega \to \Delta(Z)$, which are called \emph{Anscombe--Aumann acts}. 
As usual, we endow $\mathscr{F}(\Omega,Z)$ with the product topology with the relative topology on $\Delta(Z)$ induced by the total variation norm on $\mathrm{ca}(Z)$. 

\begin{cor}\label{cor:AnscombeAumann}
Let $\mathcal{R}$ be a binary relation on the set of Anscombe--Aumann acts $\mathscr{F}(\Omega,Z)$, where $\Omega$ is a nonempty finite set and $Z$ is a separable metric space. 

Then $\mathcal{R}$ is reflexive and satisfies the indipendence axiom \eqref{eq:stronindependence} 
if and only if there exists a nonempty family $\mathscr{K}\subseteq \mathcal{P}(\mathrm{C}_b(Z\times \Omega))$ of nonempty subsets such that, for all $f,g \in \mathscr{F}(\Omega,Z)$, it holds $(f,g) \in \mathcal{R}$ if and only if 
\begin{displaymath}
\forall K \in \mathscr{K}, \exists u \in K, \quad  
\sum\nolimits_\omega\mathbf{E}_{f(\omega)}[u(\cdot,\omega)]
\ge \sum\nolimits_\omega\mathbf{E}_{g(\omega)}[u(\cdot,\omega)].
\end{displaymath}
\end{cor}
\begin{proof}
The proof goes on the same lines of Corollary \ref{cor:firstresultok},  replacing $\Delta(Z)$ with $\Delta(Z)^\Omega$ and, similarly, $X$ with $\mathrm{ca}(Z)^\Omega$. 
Since $\mathrm{ca}(Z)$ is a normed space and $\Omega$ is finite, $X^\prime$ is isometrically isomorphic to $\mathrm{C}_b(Z\times \Omega)$ by \cite[p.77, Exercise 4]{MR1070713}. 
Hence we have the dual pair $(\mathrm{ca}(Z)^\Omega, \mathrm{C}_b(Z\times \Omega))$, with duality map given by
\begin{displaymath}
\langle \mu ,u\rangle
=\sum\nolimits_{\omega}\mathbf{E}_{\mu(\omega)}[u(\cdot, \omega)].
\end{displaymath}
The conclusion follows analogously. 
\end{proof}

As a further application, given a preference relation $\mathcal{R}$ over $\ell_{\infty}$, we provide two representations of the $\mathcal{R}$-positive cone which are related to 
\cite[Lemma 6]{MR3843491}. 
\begin{cor}\label{cor:ell_infty}
Let $\mathcal{R}$ be a binary relation of $\ell_\infty$ which satisfies\emph{:}
\begin{enumerate}[label={\rm (\roman{*})}]
\item \label{item:ellinfty1} \emph{positive homogeneity}, i.e., for all $x,y \in \ell_\infty$ and $\alpha \ge 0$, $(x,y) \in \mathcal{R}$ implies $(\alpha x,\alpha y) \in \mathcal{R}$\emph{;}
\item \label{item:ellinfty2} \emph{weak continuity}, i.e., for all $y \in \ell_\infty$, the 
sets $\mathcal{U}_y:=\{x \in \ell_\infty: (x,y) \in \mathcal{R}\}$ and $\mathcal{L}_y:=\{x \in \ell_\infty: (y,x) \in \mathcal{R}\}$ are weakly closed.
\end{enumerate}
Then there exists a nonempty family $\mathscr{K}\subseteq \mathcal{P}(\mathrm{ba}(\mathbf{N}))$\footnote{Recall that $\mathrm{ba}(\mathbf{N})$ is the family of signed finite additive measures $\mu: \mathcal{P}(\mathbf{N}) \to \mathbf{R}$ of bounded variation.} of nonempty compact convex subsets such that 
$$
\mathcal{U}_0=\left\{x\in \ell_\infty: \forall K \in \mathscr{K}, \exists \mu \in K, \,\, \int x\,\mathrm{d}\mu \ge 0\right\}.
$$
If, in addition, $\mathcal{R}$ is\emph{:}
\begin{enumerate}[label={\rm (\roman{*})}]
\setcounter{enumi}{2}
\item \label{item:ellinfty3} \emph{convex}, i.e., for all $\alpha \in \mathbf{R}$ and $x \in \ell_\infty$, the set $\{x \in \ell_\infty: (x,(\alpha,\alpha,\ldots)) \in \mathcal{R}\}$ is convex\emph{,}
\end{enumerate}
then there exists a nonempty compact convex $K\subseteq \mathrm{ba}(\mathbf{N})$ such that 
$$
\mathcal{U}_0=\left\{x\in \ell_\infty:  \forall \mu \in K, \,\, \int x\,\mathrm{d}\mu \ge 0\right\}.
$$
\end{cor}
\begin{proof}
If $\mathcal{R}$ satisfies \ref{item:ellinfty1} and \ref{item:ellinfty2}, then $(0,0) \in \mathcal{R}$, hence $\mathcal{U}_0$ is a closed pointed cone. The claim follows that Theorem \ref{thm:closedcone}, Remark \ref{rmk:closedconvexforclosedcones_Banach}, and the fact that $\ell_\infty^\prime=\mathrm{ba}(\mathbf{N})$. The second part follows similarly, using Theorem \ref{thm:dualdualcone}. 
\end{proof}



\bibliographystyle{amsplain}
\bibliography{utility}


\end{document}